\documentclass[11pt]{article}

\usepackage{geometry,amsmath,amsthm,ascmac,amssymb}
\usepackage{color,lastpage,float,fancybox,framed,natbib}

\numberwithin{equation}{section}
\newtheorem{thm}{Theorem}[section]
\newtheorem{cor}[thm]{Corollary}
\newtheorem{lem}[thm]{Lemma}

\newtheorem{rem}{Remark}

\title{
Error bounds for the normal approximation to \\
the length of a Ewens partition\thanks{This work was partly supported by Japan Society for the Promotion of Science KAKENHI Grant Number 16H02791, 18K13454.}}
\author{Koji Tsukuda\thanks{Graduate School of Arts and Sciences, The University of Tokyo, 3-8-1 Komaba, Meguro-ku, Tokyo 153-8902, Japan.}}

\date{}

\begin{document}
\maketitle

\begin{abstract}
Let $K(=K_{n,\theta})$ be a positive integer-valued random variable whose distribution is given by ${\rm P}(K = x) = \bar{s}(n,x) \theta^x/(\theta)_n$ $(x=1,\ldots,n) $, where $\theta$ is a positive number, $n$ is a positive integer, $(\theta)_n=\theta(\theta+1)\cdots(\theta+n-1)$ and $\bar{s}(n,x)$ is the coefficient of $\theta^x$ in $(\theta)_n$ for $x=1,\ldots,n$.
This formula describes the distribution of the length of a Ewens partition, which is a standard model of random partitions.
As $n$ tends to infinity, $K$ asymptotically follows a normal distribution.
Moreover, as $n$ and $\theta$ simultaneously tend to infinity, if $n^2/\theta\to\infty$, $K$ also asymptotically follows a normal distribution.
In this paper, error bounds for the normal approximation are provided.
The result shows that the decay rate of the error changes due to asymptotic regimes. 
\end{abstract}

\section{Introduction}\label{sec1}
Consider a nonnegative integer-valued random variable $K (=K_{n,\theta})$ that follows
\begin{equation} \label{FFD}
{\rm P}(K = x) = \frac{\bar{s}(n,x) \theta^x}{(\theta)_n} \quad (x=1,\ldots,n) ,
\end{equation}
where $\theta$ is a positive value, $n$ is a positive integer, $(\theta)_n=\theta(\theta+1)\cdots(\theta+n-1)$ and $\bar{s}(n,x)$ is the coefficient of $\theta^x$ in $(\theta)_n$.
This distribution is known as the falling factorial distribution \citep[equation (2.22)]{RefW}, STR1F (i.e., the Stirling family of distributions with finite support related to the Stirling number of the first kind) \citep{RefS0,RefS}, and the Ewens distribution \citep{RefKMS}.
The formula \eqref{FFD} describes the distribution of the length of a Ewens partition, which is a standard model of random partitions.
A random partition is called a Ewens partition when the distribution of the partition is given by the Ewens sampling formula. 
The Ewens sampling formula and \eqref{FFD} appear in a lot of scientific fields and have been extensively studied; see, e.g., \citet[Chapter 41]{RefJKB} or \cite{RefC}.
In the context of population genetics, \eqref{FFD} was discussed in \cite{RefE} as the distribution of the number of allelic types included in a sample of size $n$ from the infinitely-many neutral allele model with scaled mutation rate $\theta$; see also \citet[Section 1.3]{RefD}.
Moreover, in the context of nonparametric Bayesian inference, \eqref{FFD} describes the law of the number of distinct values in a sample from the Dirichlet process; see, e.g., \citet[Section 4.1]{RefGV}.
Furthermore, as introduced in \cite{RefS0}, \eqref{FFD} relates to several statistical or combinatorial topics such as permutations, sequential rank order statistics and binary search trees.

Simple calculations imply that
\[ {\rm E}[K] = \theta \sum_{i=1}^n \frac{1}{\theta +i -1}  , \quad {\rm var}(K) = \theta \sum_{i=1}^n \frac{i-1}{(\theta +i -1)^2}, \]
and
\begin{equation} \label{logapp}
{\rm E}[K] \sim {\rm var}(K)  \sim \theta\log{n} 
\end{equation}
as $n\to\infty$.
Let $\tilde{F}_{n,\theta} (\cdot)$ be the distribution function of the random variable
\[ Z_{n,\theta} = \frac{K - \theta \log{n} }{\sqrt{\theta \log{n}} } \]
standardized by the leading terms of the mean and variance, and $\Phi(\cdot)$ be the distribution function of the standard normal distribution.
By calculating the moment generating function of $Z_{n,\theta}$, \cite{RefW1} proved that $Z_{n,\theta}$ converges in distribution to the standard normal distribution; that is, $ \tilde{F}_{n,\theta} (x) \to \Phi (x) $ as $n\to\infty$ for any $x\in\mathbb{R}$.
For the history concerning this result, we refer readers to \citet[Remark after Theorem 3]{RefAT}.
In particular, when $\theta=1$ \cite{RefG} proved that $ \tilde{F}_{n,1} (x) \to \Phi (x) $ for any $x\in\mathbb{R}$.
From a theoretical perspective, it is important to derive error bounds for the approximation.
\cite{RefY} discussed the first-order Edgeworth expansion of $\tilde{F}_{n,\theta}(\cdot)$ via the Poisson approximation \citep[Remark after Theorem 3]{RefAT} and proved that
$\| \tilde{F}_{n,\theta} - \Phi \|_\infty = O \left(1/\sqrt{\log{n}} \right)$,
where $\|\cdot\|_\infty$ is the $\ell^\infty$-norm defined by
\[\| f \|_\infty = \sup_{x\in \mathbb{R}} |f(x)| \]
for a bounded function $f:\mathbb{R}\ni x\mapsto f(x)$.
Note that when $\theta=1$, \citet[Example 1]{RefH} showed that $\| \tilde{F}_{n,1} - \Phi \|_\infty =O \left(1/\sqrt{\log{n}} \right)$.
\cite{RefKMS} derived the Edgeworth expansion of the probability function of $K$, and provided the first-order Edgeworth expansion of $\tilde{F}_{n,\theta}(\cdot)$.

As the standardization of $Z_{n,\theta}$ comes from \eqref{logapp}, the normal approximation only works well when $n$ is sufficiently large with respect to $\theta$.
However, this assumption has limited validity in practical cases, so it is important to consider alternative standardized variables; see, e.g., \cite{RefY} and \cite{RefYNT}.
In particular, we consider the random variables $X_{n,\theta} $ and $Y_{n,\theta} $ defined by
\[ {X}_{n,\theta} = \frac{K - \mu_0 }{\sigma_0}
\quad {\rm and } \quad
 {Y}_{n,\theta} = \frac{K - \mu_T}{\sigma_T } ,\]
where 
\[ \mu_0 = {\rm E}[K], \quad
\sigma_0^2 = {\rm var}(K) ,
 \]
\[ \mu_T = \theta \log\left(1+\frac{n}{\theta}\right),
\quad {\rm and} \quad
\sigma_T^2 = \theta \left( \log\left(1+ \frac{n}{\theta} \right) + \frac{\theta}{n+\theta} -1 \right). \]
These are standardized random variables that use the exact moments and approximate moments, respectively.
Denote the distribution functions of $X_{n,\theta}$ and $Y_{n,\theta}$ by $F_{n,\theta}(\cdot)$ and $G_{n,\theta}(\cdot)$, respectively.
Then, \citet[Theorem 2 and Remark 6]{RefT} proved that, under the asymptotic regime $n^2/\theta\to\infty$ and $\theta\not\to0$ as $n\to\infty$ (see subsection \ref{ss:11} for the explicit assumptions), both $F_{n,\theta}(x)$ and $G_{n,\theta}(x)$ converge to $\Phi(x)$ as $n\to\infty$ for any $x\in \mathbb{R}$.
The problem considered in this paper is to provide upper and lower bounds for the approximation errors $\| F_{n,\theta} - \Phi \|_\infty $ and $\| G_{n,\theta} - \Phi \|_\infty $.

\begin{rem}
It holds that $\mu_0\sim\mu_T$ and $\sigma_0\sim\sigma_T$ as $n\to\infty$ with $n^2/\theta\to\infty$.
\end{rem}

\subsection{Assumptions and asymptotic regimes}\label{ss:11}

As explained in the Introduction, the regime $n\to\infty$ with fixed $\theta$ is sometimes unrealistic.
Hence, we consider asymptotic regimes in which $\theta$ increases as $n$ increases.
Such regimes have been discussed in \citet[Section 4]{RefF} and \cite{RefT,RefT2}.
We follow these studies.
In this subsection, let us summarize the assumptions on $n$ and $\theta$.

First, $\theta$ is assumed to be nondecreasing with respect to $n$.
Moreover, when we take the limit operation, $n^2/\theta\to\infty$ is assumed.

The following asymptotic regimes are discussed in this paper:
\begin{itemize}
\item Case A: $n/\theta\to\infty$
\item Case B: $n/\theta\to c$, where $0<c<\infty$
\item Case C: $n/\theta\to0$
\item Case C1: $n/\theta\to0$ and $n^2/\theta\to\infty$
\end{itemize}

\begin{rem}
\cite{RefF} was apparently the first to consider the asymptotic regimes in which $n$ and $\theta$ simultaneously tend to infinity.
Specifically, Cases A, B, and C were considered by \citet[Section 4]{RefF}.
Case C1 was introduced by \cite{RefT}.
\end{rem}

Furthermore, let $c^\star$ be the unique positive root of the equation
\begin{equation}\label{caseB0}
\log\left(1+x \right) - 2 + \frac{3}{x+1} - \frac{1}{(x+1)^2} = 0.
\end{equation}
Then, we introduce a new regime, Case B$^\star$, as follows:
\begin{itemize}
\item Case B$^\star$: $n/\theta\to c$, where $0<c<\infty$ and $c\neq c^\star$.
\end{itemize}

\begin{rem}
Solving \eqref{caseB0} numerically gives $c^\star = 2.16258 \cdots$. 
\end{rem}

\section{Main results}\label{sec:2}
This section presents Theorems \ref{thm1} and \ref{thm2} which are the main results of this paper and their corollaries.
Proofs of the results in this section are provided in Section \ref{sec:4}.

\subsection{An upper error bound}\label{ss21}
In this subsection, an upper bound for the error $\| F_{n,\theta} - \Phi \|_\infty$ is given in Theorem \ref{thm1}, and its convergence rate is given in Corollary \ref{cor1}.
Moreover, the convergence rate of the upper bound for the error $\| G_{n,\theta} - \Phi \|_\infty$ is given in Corollary \ref{cor2}.

We now present the first main theorem of this paper.

\begin{thm}\label{thm1}
Assume that there exists $n_0(=n_0(\theta))$ such that
\begin{equation}
 \theta \left( \log\left(1+\frac{n}{\theta}\right) - 1 + \frac{\theta}{n+\theta} \right) + \frac{n}{2(\theta+n)} - 1 >0 \label{assth1}
\end{equation}
for all $n\geq n_0$.
Then, it holds that
\[ \| F_{n,\theta} - \Phi \|_\infty \leq C \gamma_1 \]
for all $n \geq n_0$, where $C$ is a constant not larger than 0.5591 and 
\begin{equation}\label{gam1}
\gamma_1 =  
\frac{ \theta\left\{ \log\left( 1+ \frac{n}{\theta} \right) - \frac{5}{3} + \frac{3\theta}{n+\theta} - \frac{2\theta^2}{(n+\theta)^2} + \frac{2\theta^3}{3(n+\theta)^3}  \right\} + 4 + \frac{n}{n+\theta}}
{\left\{ \theta \left( \log\left(1+\frac{n}{\theta}\right) - 1 + \frac{\theta}{n+\theta} \right) + \frac{n}{2(\theta+n)} - 1  \right\}^{3/2}}. 
\end{equation}
\end{thm}

\begin{rem}
Under our asymptotic regime ($n^2/\theta\to\infty$), \eqref{assth1} is valid for sufficiently large $n$.
\end{rem}

\begin{rem}
The constant $C$ in Theorem \ref{thm1} is the universal constant appearing in the Berry--Esseen theorem.
\end{rem}

Theorem \ref{thm1} and asymptotic evaluations of the numerator and denominator of $\gamma_1$ yield the following corollary.

\begin{cor}\label{cor1}
In Cases A, B, and C1, it holds that
\begin{eqnarray*} 
\| F_{n,\theta} - \Phi \|_\infty
&=& O\left( \frac{1}{\sqrt{\theta\{\log(1+n/\theta) - 1 + \theta/(n+\theta) \}}} \right)\\
&=&\begin{cases}
  O\left(1/\sqrt{\theta\log\left({n}/{\theta}\right)} \right) & {\rm (Case \ A), } \\
  O\left(1/\sqrt{\theta}\right) & {\rm (Case \ B), } \\
  O\left(1/\sqrt{n^2/\theta}\right) & {\rm (Case \ C1).}
\end{cases} 
\end{eqnarray*}
\end{cor}

Using Corollary \ref{cor1}, we can obtain the following convergence rate of the error bound for the normal approximation to $Y_{n,\theta}$.

\begin{cor}\label{cor2}
It holds that
\[ \| G_{n,\theta} - \Phi \|_\infty =
\begin{cases}
  O\left(1/\sqrt{\theta\log\left({n}/{\theta}\right)}\right) & {\rm (Case \ A), } \\
  O\left(1/\sqrt{\theta}\right) & {\rm (Case \ B), } \\
  O\left(1/\sqrt{n^2/\theta}\right) & {\rm (Case \ C1).}
\end{cases} 
\]
\end{cor}

\subsection{Evaluation of the decay rate}\label{ss22}

In this subsection, a lower bound for the error $\| F_{n,\theta} - \Phi \|_\infty$ is given in Theorem \ref{thm2}.
Together with Theorem \ref{thm1}, this theorem yields the decay rate of $\| F_{n,\theta} - \Phi \|_\infty$, as stated in Corollary \ref{cor3}.

We now present the second main theorem of this paper.

\begin{thm}\label{thm2}
(i) Assume that there exists $n_1(=n_1(\theta))$ such that, for all $n\geq n_1$, \eqref{assth1}, ${\rm var}(K) \geq 1$ and
\begin{equation}
\theta\left\{ \log\left(1+\frac{n}{\theta}\right) -2 + \frac{3\theta}{n+\theta} - \frac{\theta^2}{(n+\theta)^2} \right\} -3 + \frac{n}{2(n+\theta)} > 0. \label{assth2i}
\end{equation}
Then, it holds that
\[ \| F_{n,\theta} - \Phi \|_\infty 
\geq \frac{\gamma_2}{D} - \gamma_3
\]
for all $n\geq n_1$, where $D$ is some constant,
\begin{equation}\label{gam2}
\gamma_2 
= \frac{\theta\left\{ \log\left(1+\frac{n}{\theta}\right) -2 + \frac{3\theta}{n+\theta} - \frac{\theta^2}{(n+\theta)^2} \right\} -3 + \frac{n}{2(n+\theta)} }{\left\{  \theta \left( \log\left(1+\frac{n}{\theta}\right) - 1 + \frac{\theta}{n+\theta} \right) + \frac{n}{\theta+n}  \right\}^{3/2}} ,
\end{equation}
and
\begin{equation}
\gamma_3 
= \frac{\theta\left\{ \frac{1}{3} - \frac{\theta}{n+\theta} + \frac{\theta^2}{(n+\theta)^2} - \frac{\theta^3}{3(n+\theta)^3}  \right\} + 2}{ \left\{ \theta \left( \log\left(1+\frac{n}{\theta}\right) -1 + \frac{\theta}{n+\theta} \right) + \frac{n}{2(\theta+n)} - 1  \right\}^2  }.  \label{gam3}
\end{equation}

(ii) Assume that there exists $n_2(=n_2(\theta))$ such that, for all $n\geq n_2$, \eqref{assth1}, ${\rm var}(K) \geq 1$ and
\begin{equation}
\theta\left\{ \log\left(1+\frac{n}{\theta}\right) -2 + \frac{3\theta}{n+\theta} - \frac{\theta^2}{(n+\theta)^2} \right\} + 2 + \frac{n}{n+\theta} < 0. \label{assth2ii}
\end{equation}
Then, it holds that
\[ \| F_{n,\theta} - \Phi \|_\infty 
\geq \frac{\gamma_4}{D} - \gamma_3
\]
for all $n\geq n_2$, where $D$ is some constant, $\gamma_3$ is as defined in \eqref{gam3}, and
\begin{equation}\label{gam4}
\gamma_4 
= \frac{- \left[ \theta\left\{ \log\left(1+\frac{n}{\theta}\right) -2 + \frac{3\theta}{n+\theta} - \frac{\theta^2}{(n+\theta)^2} \right\} + 2 + \frac{n}{n+\theta} \right] }{\left\{  \theta \left( \log\left(1+\frac{n}{\theta}\right) - 1 + \frac{\theta}{n+\theta} \right) + \frac{n}{\theta+n}  \right\}^{3/2}} .
\end{equation}
\end{thm}

\begin{rem}
Under our asymptotic regime ($n^2/\theta\to\infty$), ${\rm var}(K) \geq 1$ is valid for sufficiently large $n$.
In {\rm Case A}, \eqref{assth2i} is valid for sufficiently large $n$. 
In {\rm Case B$^\star$}, if $c > c^\star$ then \eqref{assth2i} is valid for sufficiently large $n$, and if $c < c^\star$ then \eqref{assth2ii} is valid for sufficiently large $n$.
In {\rm Case C1}, \eqref{assth2ii} is valid for sufficiently large $n$.
\end{rem}

\begin{rem}
The constant $D$ in Theorem \ref{thm2} is the universal constant introduced by \cite{RefHB}.
Note that this constant was denoted as $C$ in their theorem.
\end{rem}

As a corollary to Theorems \ref{thm1} and \ref{thm2}, we can make the following statement regarding the decay rate of $\| F_{n,\theta} - \Phi \|_\infty$.

\begin{cor}\label{cor3}
It holds that
\[ \| F_{n,\theta} - \Phi \|_\infty \asymp
\begin{cases}
   1/\sqrt{\theta\log\left({n}/{\theta}\right)} & {\rm (Case \ A) }, \\
  1/\sqrt{\theta} & {\rm (Case \ B^\star) }, \\
   1/\sqrt{n^2/\theta} & {\rm (Case \ C1)}.
\end{cases} 
\]
\end{cor}

\section{Some preliminary results}
\subsection{A representation of $K$ by a Bernoulli sequence}
Consider an independent Bernoulli random sequence $\{\xi_i \}_{i\geq1} (= \{ \xi_{i,\theta} \}_{i\geq1})$ defined by
\[ {\rm P}(\xi_i = 1) = p_i = \frac{\theta}{\theta+i-1}, \quad {\rm P}(\xi_i = 0) = 1 - p_i \quad (i=1,2,\ldots). \]
Then,
\begin{equation}  \label{berrep}
{\rm P}(K = x) = {\rm P} \left(\sum_{i=1}^n \xi_i = x \right) \quad (x=1,\ldots,n) ,
\end{equation}
that is, $\mathcal{L} (K)$ equals $\mathcal{L}(\sum_{i=1}^n \xi_i)$; see, e.g., \citet[equation (41.12)]{RefJKB} or \citet[Proposition 2.1]{RefS0}.
By virtue of this relation, and after some preparation, we will prove the results presented in Section \ref{sec:2}.
To use the Berry--Esseen-type theorem for independent random sequences (see Lemma \ref{lemTy}), we will evaluate the sum of the second- and third-order absolute central moments of $\{\xi_i\}_{i=1}^n$.
That is, we will evaluate
\begin{equation}\label{varK}
\sum_{i=1}^n {\rm E}[|\xi_i - p_i|^2] 
= \theta \sum_{i=1}^n \frac{1}{\theta+i-1} 
- \theta^2 \sum_{i=1}^n \frac{1}{(\theta+i-1)^2}  
\end{equation}
and
\begin{eqnarray}
\lefteqn{\sum_{i=1}^n {\rm E}[|\xi_i - p_i|^3]} \nonumber \\
&=& \theta \sum_{i=1}^n \frac{(i-1) \{ \theta^2 + (i-1)^2 \}}{(\theta+i-1)^4} \nonumber \\
&=&  \theta^3 \sum_{i=1}^n \frac{i-1}{(\theta+i-1)^4} +  \theta \sum_{i=1}^n \frac{(i-1)^3}{(\theta+i-1)^4}  \nonumber \\
&=& \theta \sum_{i=1}^n \frac{1}{\theta+i-1} 
- 3 \theta^2 \sum_{i=1}^n \frac{1}{(\theta+i-1)^2} 
+ 4 \theta^3 \sum_{i=1}^n \frac{1}{(\theta+i-1)^3} 
- 2 \theta^4 \sum_{i=1}^n \frac{1}{(\theta+i-1)^4}. \nonumber \\
\label{K3am}
\end{eqnarray}
To derive a lower bound result, we will evaluate
\begin{equation}
\sum_{i=1}^n {\rm E}[(\xi_i - p_i)^3]
= \theta \sum_{i=1}^n \frac{1}{\theta+i-1} 
- 3 \theta^2 \sum_{i=1}^n \frac{1}{(\theta+i-1)^2} 
+ 2 \theta^3 \sum_{i=1}^n \frac{1}{(\theta+i-1)^3}
\label{K3m}
\end{equation}
and
\begin{eqnarray}
\lefteqn{\sum_{i=1}^n \left( {\rm E}[|\xi_i - p_i|^2] \right)^2} \\
&=& \theta^2 \sum_{i=1}^n \frac{(i-1)^2}{(\theta+i-1)^4} \nonumber \\
&=& \theta^2 \sum_{i=1}^n \frac{1}{(\theta+i-1)^2} 
- 2 \theta^3 \sum_{i=1}^n \frac{1}{(\theta+i-1)^3}
+ \theta^4 \sum_{i=1}^n \frac{1}{(\theta+i-1)^4} .  \nonumber \\
\label{K2m2}
\end{eqnarray}

\begin{rem}
It follows from the binomial theorem that
\begin{eqnarray*}
\lefteqn{\sum_{i=1}^n {\rm E}\left[ (\xi_i - p_i)^m \right] }\\
&=& \sum_{j=1}^{m-1} (-1)^{j-1} \binom{m}{j-1} \sum_{i=1}^n \left( \frac{\theta}{\theta+i-1} \right)^j + (-1)^{m-1}(m-1) \sum_{i=1}^n \left( \frac{\theta}{\theta+i-1} \right)^m
\end{eqnarray*}
for any $m = 2,3,\ldots$.
\end{rem}

\subsection{Evaluations for moments}

In this subsection, we evaluate several sums of moments of $\{ \xi_i \}_{i=1}^n$.

\begin{lem}\label{Lem31}
(i) It holds that
\begin{eqnarray*}
\lefteqn{\theta \left( \log\left(1+\frac{n}{\theta}\right) -1 + \frac{\theta}{n+\theta} \right) + \frac{n}{2(\theta+n)} - 1} \\ 
&\leq& \sum_{i=1}^n {\rm E}[|\xi_i - p_i|^2]  \\
&\leq& \theta \left( \log\left(1+\frac{n}{\theta}\right) - 1 + \frac{\theta}{n+\theta} \right) + \frac{n}{\theta+n} . 
\end{eqnarray*}
(ii) If $n^2/\theta\to\infty$ then it holds that
\[ \sum_{i=1}^n {\rm E}[|\xi_i - p_i|^2]  \sim \theta \left( \log\left(1+\frac{n}{\theta}\right) - 1 + \frac{\theta}{n+\theta} \right). \]
(iii)
In particular, it holds that
\[ \sum_{i=1}^n {\rm E}[|\xi_i - p_i|^2]  \sim 
\begin{cases}
  \theta \log\left(\frac{n}{\theta}\right) & {\rm (Case \ A) }, \\
  \theta \left( \log\left(1+c \right) - 1 + \frac{1}{c+1} \right) & {\rm (Case \ B) }, \\
   \frac{n^2}{2\theta} & {\rm (Case \ C)}.
\end{cases} 
\]
\end{lem}

\begin{proof}[Proof of Lemma \ref{Lem31}]
(i) The desired inequality is an immediate consequence of \eqref{varK} and Lemma~\ref{LemA}.
(ii) As 
\[\log\left(1+x \right) - 1 + \frac{1}{x+1} >0 \]
for any $x>0$, it holds that
\[ \theta \left\{ \log\left(1+\frac{n}{\theta}\right) - 1 + \frac{\theta}{n+\theta} \right\} \to \infty, \]
whereas the remainder does not diverge to $\pm\infty$.
This implies the assertion.
(iii) The assertion is a direct consequence of (ii) (for Case C, the result follows from the Taylor expansion of $\log\left(1+x \right) - 1 + 1/(x+1)$ as $x\to0$).
\end{proof}

\begin{lem}\label{Lem32}
(i) It holds that
\begin{eqnarray*}
\lefteqn{\theta\left\{ \log\left( 1+ \frac{n}{\theta} \right) - \frac{5}{3} + \frac{3\theta}{n+\theta} - \frac{2\theta^2}{(n+\theta)^2} + \frac{2\theta^3}{3(n+\theta)^3}  \right\} + \frac{n}{2(n+\theta)} -5 }\\
&\leq& \sum_{i=1}^n {\rm E}[|\xi_i - p_i|^3]   \\
&\leq& \theta\left\{ \log\left( 1+ \frac{n}{\theta} \right) - \frac{5}{3} + \frac{3\theta}{n+\theta} - \frac{2\theta^2}{(n+\theta)^2} + \frac{2\theta^3}{3(n+\theta)^3}  \right\} + 4 + \frac{n}{n+\theta}.
\end{eqnarray*}
(ii) If $n^2/\theta\to\infty$, then it holds that
\[ \sum_{i=1}^n {\rm E}[|\xi_i - p_i|^3]  
\sim \theta\left\{ \log\left( 1+ \frac{n}{\theta} \right) - \frac{5}{3} + \frac{3\theta}{n+\theta} - \frac{2\theta^2}{(n+\theta)^2} + \frac{2\theta^3}{3(n+\theta)^3}  \right\}. \]
(iii) In particular, it holds that
\[ \sum_{i=1}^n {\rm E}[|\xi_i - p_i|^3]  \sim 
\begin{cases}
  \theta \log\left(\frac{n}{\theta}\right) & {\rm (Case \ A) }, \\
  \theta \left\{ \log\left(1+c \right) - \frac{5}{3} + \frac{3}{c+1} - \frac{2}{(c+1)^2} + \frac{2}{3(c+1)^3} \right\} & {\rm (Case \ B) } , \\
   \frac{n^2}{2\theta} & {\rm (Case \ C)}.
\end{cases} 
\]
\end{lem}

\begin{proof}[Proof of Lemma \ref{Lem32}]
(i) The desired inequality is an immediate consequence of \eqref{K3am} and Lemma~\ref{LemA}.
(ii) As
\[ \log\left(1+x \right) - \frac{5}{3} + \frac{3}{x+1} - \frac{2}{(x+1)^2} + \frac{2}{3(x+1)^3} >0 \]
for any $x>0$, it holds that
\[ \theta\left\{ \log\left( 1+ \frac{n}{\theta} \right) - \frac{5}{3} + \frac{3\theta}{n+\theta} - \frac{2\theta^2}{(n+\theta)^2} + \frac{2\theta^3}{3(n+\theta)^3}  \right\} \to \infty, \]
whereas the remainder does not diverge to $\pm\infty$.
This implies the assertion.
(iii) The assertion is a direct consequence of (ii) (for Case C, the result follows from the Taylor expansion of $\log\left(1+x \right) - 5/3 + 3/(x+1) - 2/(x+1)^2 + 2/\{3(x+1)^3\}$ as $x\to0$).
\end{proof}

\begin{lem}\label{Lem33}
(i) It holds that
\begin{eqnarray*}
\lefteqn{\theta\left\{ \log\left( 1+ \frac{n}{\theta} \right) - 2 + \frac{3\theta}{n+\theta} - \frac{\theta^2}{(n+\theta)^2}  \right\} + \frac{n}{2(n+\theta)} - 3 }\\
&\leq& \sum_{i=1}^n {\rm E}[(\xi_i - p_i)^3]   \\
&\leq& \theta\left\{ \log\left( 1+ \frac{n}{\theta} \right) - 2 + \frac{3\theta}{n+\theta} - \frac{\theta^2}{(n+\theta)^2}  \right\} + 2 + \frac{n}{n+\theta}.
\end{eqnarray*}
(ii) In Case A, B$^\star$, or C, it holds that
\[ \sum_{i=1}^n {\rm E}[(\xi_i - p_i)^3]  \sim 
\begin{cases}
  \theta \log\left(\frac{n}{\theta}\right) & {\rm (Case \ A) ,} \\
  \theta \left\{ \log\left(1+c \right) - 2 + \frac{3}{c+1} - \frac{1}{(c+1)^2}  \right\} &  {\rm (Case \ B^\star),} \\
  -\frac{n^2}{2\theta} & {\rm (Case \ C).}
\end{cases} 
\]
\end{lem}

\begin{proof}[Proof of Lemma \ref{Lem33}]
(i) The desired inequality is an immediate consequence of \eqref{K3m} and Lemma~\ref{LemA}.
(ii) In Case A, the assertion holds because 
\[ \theta\left\{ \log\left( 1+ \frac{n}{\theta} \right) - 2 + \frac{3\theta}{n+\theta} - \frac{\theta^2}{(n+\theta)^2}  \right\} 
\sim \theta \log\left( 1+ \frac{n}{\theta} \right) \to \infty, \]
whereas the remainder does not diverge to $\pm \infty$.
In Case B$^\star$, the assertion holds because
\[  \log\left( 1+ \frac{n}{\theta} \right) - 2 + \frac{3\theta}{n+\theta} - \frac{\theta^2}{(n+\theta)^2}  
\to \log\left(1+c \right) - 2 + \frac{3}{c+1} - \frac{1}{(c+1)^2} \neq 0 \]
and $\theta\to\infty$, whereas the remainder does not diverge to $\pm \infty$.
In Case C, the assertion holds because 
\[ \theta\left\{ \log\left( 1+ \frac{n}{\theta} \right) - 2 + \frac{3\theta}{n+\theta} - \frac{\theta^2}{(n+\theta)^2}  \right\} 
\sim \frac{-n^2}{2\theta}
\to -\infty, \]
whereas the remainder terms do not diverge to $\pm \infty$.
\end{proof}

\begin{lem}\label{Lem34}
It holds that
\begin{equation}
\sum_{i=1}^n \left( {\rm E}[|\xi_i - p_i|^2]\right)^2 
\leq \theta\left\{ \frac{1}{3} - \frac{\theta}{n+\theta} + \frac{\theta^2}{(n+\theta)^2} - \frac{\theta^3}{3(n+\theta)^3}  \right\} + 2.  \label{ev22}
\end{equation}
\end{lem}

\begin{proof}[Proof of Lemma \ref{Lem34}]
The assertion is an immediate consequence of \eqref{K2m2} and Lemma~\ref{LemA}.
\end{proof}

\begin{rem}
The asymptotic value of the RHS in \eqref{ev22} is given by
$\theta/3$ (Case A), \\ $\theta\left[ {1}/{3} - {1}/{(c+1)} + {1}/{(c+1)^2} - 1/\{3(c+1)^3\}  \right]$ (Case B), or ${n^3}/{(3\theta^2)} + 2$ (Case C).
\end{rem}

\section{Proofs of the results in Section \ref{sec:2}}
\label{sec:4}

\subsection{Proof of the results in Subsection \ref{ss21}}

In this subsection, we provide proofs of the results in Subsection \ref{ss21}.

\begin{proof}[Proof of Theorem \ref{thm1}]
Let $n$ be an arbitrary integer such that $n \geq n_0$.
From \eqref{berrep}, Lemma \ref{lemTy} yields that
\[ \| F_{n,\theta} - \Phi \|_\infty \leq C \frac{ \sum_{i=1}^n {\rm E}[|\xi_i - p_i|^3] }{(\sum_{i=1}^n {\rm E}[|\xi_i - p_i|^2])^{3/2}} , \]
where $C$ is the constant appearing in Lemma \ref{lemTy}.
Additionally, Lemmas \ref{Lem31}-(i) and \ref{Lem32}-(i) yield that
\[ \frac{ \sum_{i=1}^n {\rm E}[|\xi_i - p_i|^3] }{(\sum_{i=1}^n {\rm E}[|\xi_i - p_i|^2])^{3/2}} \leq \gamma_1 .\]
\end{proof}

\begin{proof}[Proof of Corollary \ref{cor1}]
In Case A, B, or C1, it holds that
\[
\theta \left\{ \log\left(1+\frac{n}{\theta}\right) - \frac{5}{3} + \frac{3\theta}{n+\theta} - \frac{2\theta^2}{(n+\theta)^2} + \frac{2\theta^3}{3(n+\theta)^3} \right\} 
\asymp
\theta \left( \log\left(1+\frac{n}{\theta}\right) - 1 + \frac{\theta}{n+\theta} \right) .
\]
Hence, Theorem \ref{thm1}, Lemmas \ref{Lem31} and \ref{Lem32} yield that
\[\gamma_1=O\left( \frac{1}{\sqrt{\theta \left\{ \log\left(1+n/\theta \right) - 1 + \theta/(n+\theta) \right\} }} \right). \]
This completes the proof.
\end{proof}

\begin{proof}[Proof of Corollary \ref{cor2}]
From
\[ G_{n,\theta}(x)=F_{n,\theta} \left(\frac{\sigma_T}{\sigma_0}x + \frac{\mu_T - \mu_0}{\sigma_0} \right) \quad (x\in\mathbb{R}) \]
and the triangle inequality, it follows that
\begin{eqnarray}
\| G_{n,\theta} - \Phi \|_\infty 
&=& \sup_{x\in \mathbb{R}} \left| F_{n,\theta} \left(\frac{\sigma_T}{\sigma_0}x + \frac{\mu_T - \mu_0}{\sigma_0} \right) - \Phi(x) \right| \nonumber \\
&\leq& \sup_{x\in \mathbb{R}} \left| F_{n,\theta} \left(\frac{\sigma_T}{\sigma_0}x + \frac{\mu_T - \mu_0}{\sigma_0} \right) - \Phi\left(\frac{\sigma_T}{\sigma_0}x + \frac{\mu_T - \mu_0}{\sigma_0} \right) \right| \nonumber \\
&&+ \sup_{x\in \mathbb{R}} \left| \Phi \left(\frac{\sigma_T}{\sigma_0}x + \frac{\mu_T - \mu_0}{\sigma_0} \right) - \Phi\left(\frac{\sigma_T}{\sigma_0}x  \right) \right| 
+ \sup_{x\in \mathbb{R}} \left|  \Phi\left(\frac{\sigma_T}{\sigma_0}x  \right) - \Phi\left(x  \right) \right| . \nonumber \\ \label{pc2t1}
\end{eqnarray}
The first term on the RHS in \eqref{pc2t1} is 
\[ O\left( \frac{1}{\sqrt{\theta\{\log(1+n/\theta) - 1 + \theta/(n+\theta) \}}} \right) \]
from Corollary \ref{cor1}.
The second term on the RHS in \eqref{pc2t1} is bounded above by
\[ \frac{1}{\sqrt{2\pi}} \frac{|\mu_T-\mu_0|}{\sigma_0} = O\left( \frac{1}{\sqrt{\theta \{ \log(1+n/\theta) - 1 + \theta/(n+\theta)\} }} \right) \]
from Lemma \ref{LemA2}-(i).
This is because $|\mu_T-\mu_0| = O(1)$ (Lemma \ref{LemA}) and $\sigma_0^2 \sim \theta (\log(1+n/\theta) - 1 + \theta/(n+\theta))$ (Lemma \ref{Lem31}).
The third term of the RHS in \eqref{pc2t1} is bounded above by
\[ 
\frac{1}{\sqrt{2\pi {\rm e}}} \max\left( \frac{\sigma_T}{\sigma_0} ,1 \right) \left| 1 -  \frac{\sigma_0}{\sigma_T} \right|
=O\left( \frac{1}{\theta \{ \log(1+n/\theta) - 1 + \theta/(n+\theta) \} } \right) \]
from Lemma \ref{LemA2}-(ii).
This is because, from $\sigma_0^2 \geq 1 \geq n/(n+\theta)$ for $n \geq n_1$ and
\[ \sigma_0^2 - \frac{n}{\theta+n} \leq \sigma_T^2 \leq \sigma_0^2 - \frac{n}{2(\theta+n)} +1 \]
(see Lemma \ref{Lem31}-(i)), it follows that
\begin{equation}
\left( 1- \frac{1}{\sigma_0^2} \frac{n}{\theta+n } \right)^{1/2}  
\leq \frac{\sigma_T}{\sigma_0} 
\leq \left[ 1 + \frac{1}{\sigma_0^2} \left\{1 - \frac{n}{2 (\theta+n) } \right\} \right]^{1/2}  \label{prno}
\end{equation}
for $n\geq n_1$.
Note that the LHS and RHS of \eqref{prno} are
\[ 1+ O \left( \frac{1}{\sigma_0^2} \right) 
= 1+ O\left( \frac{1}{\theta \{ \log(1+n/\theta) - 1 + \theta/(n+\theta) \}} \right) . \]
This completes the proof.
\end{proof}

\subsection{Proof of the results in Subsection \ref{ss22}}

In this subsection, we provide proofs of the results in Subsection \ref{ss22}.

\begin{proof}[Proof of Theorem \ref{thm2}]
(i) Let $n$ be an arbitrary integer such that $n \geq n_1$.
As $|\xi_i - p_i| <1 \leq {\rm var}(K)$ for all $i=1,\ldots,n$, \eqref{berrep} and Lemma \ref{lemHB} yield that
\begin{equation} \label{aplHB}
\| F_{n,\theta} - \Phi \|_\infty 
\geq \frac{1}{D} \frac{ \left| \sum_{i=1}^n {\rm E}[ (\xi_i - p_i)^3] \right| }{ \left( \sum_{i=1}^n {\rm E}[ |\xi_i - p_i|^2]\right)^{3/2} } 
- \frac{\sum_{i=1}^n \left( {\rm E}[ (\xi_i - p_i)^2]\right)^2}{\left( \sum_{i=1}^n {\rm E}[ |\xi_i - p_i|^2]\right)^{2}},  
\end{equation}
where $D$ is the constant appearing in Lemma \ref{lemHB}.
Additionally, Lemmas \ref{Lem31}-(i) and \ref{Lem33}-(i) yield that 
\[ \frac{ \sum_{i=1}^n {\rm E}[ (\xi_i - p_i)^3] }{ \left( \sum_{i=1}^n {\rm E}[ |\xi_i - p_i|^2]\right)^{3/2} } \geq \gamma_2 >0 .\]
Moreover, Lemmas \ref{Lem31}-(i) and \ref{Lem34}-(i) yield that 
\begin{equation} \label{aplL34}
 \frac{\sum_{i=1}^n \left( {\rm E}[ (\xi_i - p_i)^2]\right)^2}{\left( \sum_{i=1}^n {\rm E}[ |\xi_i - p_i|^2]\right)^{2}} \leq \gamma_3. 
\end{equation}
This completes the proof of (i).

(ii) Let $n$ be an arbitrary integer such that $n \geq n_2$.
From the same reason as (i), \eqref{berrep} and Lemma \ref{lemHB} yields \eqref{aplHB}.
Additionally, Lemmas \ref{Lem31}-(i) and \ref{Lem33}-(i) yield that 
\[ - \frac{ \sum_{i=1}^n {\rm E}[ (\xi_i - p_i)^3] }{ \left( \sum_{i=1}^n {\rm E}[ |\xi_i - p_i|^2]\right)^{3/2} } \geq  \gamma_4 > 0 .\]
Moreover, Lemmas \ref{Lem31}-(i) and \ref{Lem34}-(i) yield \eqref{aplL34}.
This completes the proof of (ii).
\end{proof}

\begin{proof}[Proof of Corollary \ref{cor3}]
In Case A, it follows from 
\[
\theta\left\{ \log\left( 1+ \frac{n}{\theta} \right) - 2 + \frac{3\theta}{n+\theta} - \frac{\theta^2}{(n+\theta)^2} \right\}
\sim \theta \left( \log\left(1+\frac{n}{\theta}\right) - 1 + \frac{\theta}{n+\theta} \right)  
\sim \theta \log\left(\frac{n}{\theta}\right) 
\]
that
\[ \gamma_2 = O\left( \frac{1}{\sqrt{\theta\log(n/\theta)}} \right) . \]
Moreover, it holds that
\[ \gamma_3 = O\left( \frac{1}{\theta (\log(1+n/\theta))^2} \right) .\]
Hence, Corollary \ref{cor1} and Theorem \ref{thm2} yield the desired result in Case A.

In Case B$^\star$, it follows from
\[
\theta \left| \log\left( 1+ \frac{n}{\theta} \right) - 2 + \frac{3\theta}{n+\theta} - \frac{\theta^2}{(n+\theta)^2}  \right| 
\asymp \theta \left( \log\left(1+\frac{n}{\theta}\right) - 1 + \frac{\theta}{n+\theta} \right)
\asymp \theta 
\]
that
\[ \gamma_2 \sim \gamma_4 = O\left( \frac{1}{\sqrt{\theta}} \right). \]
Moreover, it holds that
\[ \gamma_3 = O\left( \frac{1}{\theta} \right) .\]
As 
\[ \theta \left| \log\left( 1+ \frac{n}{\theta} \right) - 2 + \frac{3\theta}{n+\theta} - \frac{\theta^2}{(n+\theta)^2}  \right| \to \infty ,\]
either $n_2$ or $n_3$ exist in Case B$^\star$.
Hence, Corollary \ref{cor1} and Theorem \ref{thm2} yields the desired result in Case B$^\star$.

In Case C1, it follows from 
\[
- \theta\left\{ \log\left( 1+ \frac{n}{\theta} \right) - 2 + \frac{3\theta}{n+\theta} - \frac{\theta^2}{(n+\theta)^2}  \right\} 
\sim \theta \left( \log\left(1+\frac{n}{\theta}\right) - 1 + \frac{\theta}{n+\theta} \right) 
\sim \frac{n^2}{2 \theta}
\]
that
\[ \gamma_4 = O\left( \frac{1}{ \sqrt{n^2/\theta} } \right) . \]
Moreover, it holds that
\[ \gamma_3 \sim \frac{\frac{n^3}{3\theta^2} + 2}{\frac{n^4}{4\theta^2}} = O\left( \frac{1}{n} + \frac{1}{n^4/\theta^2}  \right) . \]
Hence, Corollary \ref{cor1} and Theorem \ref{thm2} yield the desired result in Case C1.

This completes the proof.
\end{proof}

\section{Concluding remarks}

In this paper, we evaluated the approximation errors $\| F_{n,\theta} - \Phi \|_\infty $ and $\| G_{n,\theta} - \Phi \|_\infty $.
Deriving decay rates for $\| F_{n,\theta} - \Phi \|_\infty $ when $n/\theta \to c^\star$ (i.e., Case B with $c=c^\star$) and for $\| G_{n,\theta} - \Phi \|_\infty $ is left for future research.
Moreover, as normal approximations are refined by the Edgeworth expansion, it is also important to derive the Edgeworth expansion under our asymptotic regimes.


\appendix
\section{Some evaluations}

The following lemma is used in the main body.

\begin{lem}\label{LemA}
Let $\theta$ be a positive value and $n$ be a positive integer.
(i) It holds that
\[ \log\left(1+\frac{n}{\theta}\right) + \frac{n}{2\theta(n+\theta)} 
\leq \sum_{i=1}^n \frac{1}{\theta+i-1} 
\leq \log\left(1+\frac{n}{\theta}\right) + \frac{n}{\theta(n+\theta)} \]
(ii) It holds that
\[ \frac{1}{k \theta^k} - \frac{1}{k (n+\theta)^k}
\leq \sum_{i=1}^n \frac{1}{(\theta+i-1)^{k+1}} 
\leq \frac{1}{\theta^{k+1}} + \frac{1}{k \theta^k} - \frac{1}{k (n+\theta)^k} \]
for any positive integer $k$.
\end{lem}

\begin{proof}
For (i), see \citet[Proof of Proposition 1]{RefT}.
For (ii), the conclusion follows from
\[ \int_{\theta}^{n+\theta} \frac{dx}{x^{k+1}}
\leq \sum_{i=1}^n \frac{1}{(\theta+i-1)^{k+1}} 
\leq \frac{1}{\theta^{k+1}} + \int_{\theta}^{n+\theta} \frac{dx}{x^{k+1}} \]
for any positive integer $k$.
This completes the proof.
\end{proof}

The next lemma provides some basic results on the standard normal distribution function.

\begin{lem}\label{LemA2}
(i) For any $\alpha (\in\mathbb{R})$, it holds that
\[ \sup_{x\in\mathbb{R}} |\Phi(x+\alpha) - \Phi (x)| \leq \frac{|\alpha|}{\sqrt{2\pi}} .\]
(ii)  For any positive $\beta (\in\mathbb{R})$, it holds that
\[ \sup_{x\in\mathbb{R}} |\Phi(\beta x) - \Phi (x)| 
\leq
\max (\beta,1) \frac{\left| 1 - 1/\beta \right|}{\sqrt{2\pi {\rm e}}}.
\]
\end{lem}

\begin{proof}
(i) For some $\delta$ between 0 and $\alpha$, it holds that
\[ | \Phi(x+\alpha) - \Phi(x) | = \phi(x+\delta) |\alpha|  \leq \sup_{x\in\mathbb{R}} \phi(x) |\alpha| = \frac{|\alpha|}{\sqrt{2\pi}}. \]

(ii) As
\[ \max (\beta,1) \frac{\left| 1 - 1/\beta \right|}{\sqrt{2\pi {\rm e}}} =
\begin{cases}
(\beta-1)/{\sqrt{2\pi {\rm e}}} & (\beta\geq1) \\
(1/\beta - 1)/{\sqrt{2\pi {\rm e}}} & (0<\beta\leq1),
\end{cases}\]
we prove the assertion for $\beta\geq1$ and $\beta\leq1$, separately.
First, we consider the case $\beta \geq 1$.
For $x=0$, it holds that $|\Phi(\beta x) - \Phi(x)|=0$.
For $x>0$,
\begin{eqnarray*}
\lefteqn{|\Phi(\beta x) - \Phi(x)| 
= \int_x^{\beta x} \phi(t) dt} \\
&\leq& \phi(x) (\beta-1)x 
= -\phi'(x) (\beta-1) 
\leq \sup_{x>0}(-\phi'(x)) (\beta-1)
= \frac{\beta-1}{\sqrt{2\pi {\rm e}}} .
\end{eqnarray*}
For $x<0$,
\begin{eqnarray*}
\lefteqn{|\Phi(\beta x) - \Phi(x)| 
= \int_{\beta x}^x \phi(t) dt} \\
&\leq& \phi(x) (1- \beta)x 
= \phi'(x) (\beta-1) 
\leq  \sup_{x < 0}(\phi'(x)) (\beta-1)
= \frac{\beta -1 }{\sqrt{2\pi {\rm e}}} .
\end{eqnarray*}
Next, we consider the case $(0< ) \beta \leq 1$.
For $x=0$, it holds that $|\Phi(\beta x) - \Phi(x)|=0$.
For $x>0$,
\begin{eqnarray*}
\lefteqn{|\Phi(\beta x) - \Phi(x)| 
= \int_{\beta x}^x \phi(t) dt} \\
&\leq& \phi(\beta x) (1- \beta)x 
= - \phi'(\beta x) \left(\frac{1}{\beta} -1\right) 
\leq \sup_{x> 0}(- \phi'(x)) \left(\frac{1}{\beta} -1\right)
= \frac{1/\beta -1 }{\sqrt{2\pi {\rm e}}} .
\end{eqnarray*}
For $x<0$,
\begin{eqnarray*}
\lefteqn{ |\Phi(\beta x) - \Phi(x)| 
= \int_x^{\beta x} \phi(t) dt} \\
&\leq& \phi(\beta x) (\beta-1)x 
= \phi'(\beta x) \left(\frac{1}{\beta} -1 \right)
\leq \sup_{x<0}(\phi'(x)) \left(\frac{1}{\beta} -1 \right) 
= \frac{1/\beta-1}{\sqrt{2\pi {\rm e}}} .
\end{eqnarray*}

This completes the proof.
\end{proof}

\section{Error bounds for normal approximations}
\subsection{The Berry--Esseen-type theorem for independent sequences}
In this subsection, we introduce the Berry--Esseen-type theorem for independent sequences.
For further details, see \cite{RefTy}.

Let $\{ X_i\}_{i\geq1}$ be a sequence of independent random variables, and ${\rm E}[X_i]=0$, ${\rm E}[X_i^2]=\sigma_i^2 (>0)$, ${\rm E}[|X_i|^3]=\beta_i <\infty$ for all $i=1,2,\ldots$.
The quantity
$ \varepsilon_n = {\sum_{i=1}^n \beta_i}/{(\sum_{i=1}^n \sigma_i^2)^{3/2}} $
is called the Lyapunov fraction.
We denote the distribution function of 
$ {\sum_{i=1}^n X_i}/{(\sum_{i=1}^n \sigma_i^2)^{1/2}} $
by $F^X_n$.
Then, the following result holds.

\begin{lem}[\cite{RefTy}]\label{lemTy}
There exists a universal constant $C$ such that
\[ \| F^X_n - \Phi \|_\infty \leq C \varepsilon_n \]
for all positive integers $n$, where $C$ does not exceed $0.5591$.
\end{lem}

\begin{rem}
Here, we introduce the result given by \cite{RefTy}.
There have been many studies in which Berry--Esseen-type results are derived; see, e.g., \citet[Chapter 3]{RefCGS}.
\end{rem}

\subsection{Lower bound}
In this subsection, we introduce the result given by \cite{RefHB} that considers reversing the Berry--Esseen inequality.

Let $\{ Y_i\}_{i\geq1}$ be a sequence of independent random variables satisfying ${\rm E}[Y_i]=0$ and ${\rm E}[Y_i^2]=\sigma_i^2 (>0)$ for all $i$, and $\sum_{i=1}^n \sigma_i^2 =1$.
We denote the distribution function of $\sum_{i=1}^n Y_i $ by $F^Y_n$.
Letting 
\[\delta = \sum_{i=1}^n {\rm E}[Y_i^2 \mathbb{I} \{|Y_i| >1\}] + \sum_{i=1}^n {\rm E}[Y_i^4 \mathbb{I} \{|Y_i| \leq1\}] + \left| \sum_{i=1}^n {\rm E}[Y_i^3 \mathbb{I} \{|Y_i| \leq 1\}]  \right| , \]
the following result holds.

\begin{lem}[\cite{RefHB}]\label{lemHB}
There exists a universal constant $D$ such that
\[ \delta \leq D\left( \|F^Y_n - \Phi \|_{\infty} + \sum_{i=1}^n \sigma_i^4 \right). \]
\end{lem}

As
\[ \delta \geq  \sum_{i=1}^n {\rm E}[Y_i^2 \mathbb{I} \{|Y_i| >1\}] + \left| \sum_{i=1}^n {\rm E}[Y_i^3 \mathbb{I} \{|Y_i| \leq 1\}]  \right| ,  \]
we use the RHS as a lower bound.
This bound is sufficient in Cases A, B$^\star$, and C1 to show the decay rate of $\|F_{n,\theta} - \Phi \|_\infty$.

\end{document}